\documentclass{amsart}
\usepackage{graphicx}
\usepackage{amsmath,amssymb}
\newtheorem{thm}{Theorem}[section]
\newtheorem{cor}[thm]{Corollary}

\theoremstyle{definition}

\theoremstyle{remark}

\numberwithin{equation}{section}

\begin{document}

\title[Hyers-Ulam-Rassias stability]{A Fixed Points Approach to    stability of the
 Pexider  Equation }
\author[  E. Elqorachi, J. M. Rassias and B. Bouikhalene]{ E. Elqorachi, John M. Rassias and B. Bouikhalene}


\begin{abstract} Using the  fixed point theorem we establish the  Hyers-Ulam-Rassias stability  of  the generalized Pexider functional
equation $$\frac{1}{\mid K\mid}\sum_{k\in K}f(x+k\cdot
y)=g(x)+h(y),\;\;x,y\in E$$ from a normed space $E$ into a complete
$\beta$-normed  space $F$,
 where $K$ is a finite abelian subgroup of the  automorphism  group  of the group $(E,+)$.
\end{abstract}
\maketitle
\section{Introduction and Preliminaries} Under what condition does there exist a group homomorphism near an approximate group homomorphism?
This question concerning the stability of group homomorphisms was
posed by S. M. Ulam \cite{43}. In 1941, the Ulam's problem for the
case of approximately additive mappings was solved by D. H. Hyers
\cite{7} on Banach
  spaces. In 1950 T. Aoki \cite{0} provided a generalization of the
  Hyers' theorem for additive mappings and in 1978 Th. M. Rassias
  \cite{19} generalized the Hyers' theorem for linear mappings by
  considering an unbounded Cauchy difference. The result of Rassias' theorem has been generalized by J.M. Rassias \cite{12f} and later by  G\v{a}vruta \cite{Gav}
who permitted the Cauchy difference to be bounded by a general
  control function.
 Since then, the stability problems for several  functional equations have been
  extensively investigated  (cf.  \cite{0},...,  \cite{mos1}, \cite{4c},..., \cite{Rato1}, \cite{37},..., \cite{31a} and  \cite{Yang}).\\
Let $E$ be a real vector space and $F$ be a real Banach space. Let
$K$ be a finite abelian subgroup of $Aut(E)$ (the automorphism group
of the group $(E,+$), $|K|$ denotes the order of $K$. Writing the
action of $k\in{K}$ on $x\in{E}$ as $k\cdot x$, we will say that
$(f,g,h)\;:E\rightarrow F$ is a solution of the generalized Pexider
functional equation, if
\begin{equation}\label{eq11}
 \frac{1}{|K|}\sum_{k\in{K}} f(x+k\cdot y)=g(x)+h(y),\phantom{+}
    x,y\in{E}
\end{equation}
The generalized quadratic functional equation\begin{equation}\label{eq12}
 \frac{1}{|K|}\sum_{k\in{K}} f(x+k\cdot y)=f(x)+f(y),\phantom{+}
    x,y\in{E}
\end{equation} and   the generalized Jensen functional equation
\begin{equation}\label{eq13}
 \frac{1}{|K|}\sum_{k\in{K}} f(x+k\cdot y)=f(x),
    \phantom{+}x,y\in{E}
\end{equation}are  particulars cases of equation (\ref{eq11}).
\\The  functional equations (\ref{eq11}), (\ref{eq12}) and
(\ref{eq13}) appeared in several works by
H. Stetk\ae r, see for example \cite{38}, \cite{40} and \cite{41}. We refer also to the recent studies by {{\L}}. Rados{\l}aw \cite{Rato1} and \cite{Rato2}.  \\
If we set $K=\{I,\sigma\}$, were $I$: $E\longrightarrow E$ denotes
the identity function and $\sigma$ denote an additive function of
$E$, such that $\sigma(\sigma(x))=x,$ for all $x\in E$ then equation
(\ref{eq11}) reduces to the Pexider functionals equations
\begin{equation}\label{eq14}
f(x+y)+f(x+\sigma(y))=g(x)+h(y),\;x,y\in{E},
\end{equation}
\begin{equation}\label{eq15}
f(x+y)=g(x)+h(y),\;x,y\in{E},\;\;\;(\sigma=I)
\end{equation}
\begin{equation}\label{eq14}
f(x+y)+f(x-y)=g(x)+h(y),\;x,y\in{E},\;\;(\sigma=-I)
\end{equation}Y. H. Lee and K. W. Jung \cite{Lee} obtained the Hyers-Ulam-Rassias of the Pexider functional equation (\ref{eq15}). Jung \cite{10a} and Jung and Sahoo \cite{10c} investigated the Hyers-Ulam-Rassias stability of equation (\ref{eq14}).
 Belaid et al.  have proved the Hyers-Ulam stability of equation (\ref{eq11}) and the Hyers-Ulam-Rassias
stability of the  functional equations (\ref{eq12}), (\ref{eq13}), (see \cite{1}, \cite{2}, \cite{2a} and \cite{man1} ).\\
Recently, Rados{\l}aw \cite{Rato1} obtained the Hyers-Ulam-Rassias stability of equation (\ref{eq11}).
In 2003 L. C\v{a}dariu and V. Radu \cite{1a} notice that a fixed point alternative method is very important for the solution
of the Hyers-Ulam stability problem. Subsequently, this method was applied to investigate
the Hyers-Ulam-Rassias stability for Jensen functional equation, as well as for the additive Cauchy functional equation \cite{Aa}
 by considering a general control function $\varphi(x,y)$, with suitable properties, using such an elegant idea,
 several authors applied the method to investigate the stability of some functional equations, see for example  \cite{akk}, \cite{akk2},  \cite{zoom}, \cite{mos3} and \cite{12e}.\\\\
 The \textit{fixed point method} was used for the first time by J. A. Baker  \cite{bak} who applied a variant of Banach's fixed point theorem to obtain the Hyers Ulam stability of a functional equation in a single variable. For more information we refer to  the recent studies by K. Cieplin\'{s}ki \cite{Cie}. \\\\In this paper, we will apply the fixed point method as in \cite{1a} to
prove the Hyers-Ulam-Rassias stability of the functional equations
(\ref{eq11}), for a large classe of functions from a vector space $E$ into complete $\beta$-normed space $F$.\\\\ Now, we recall   one of fundamental results of fixed point theory. \\Let $X$ be a set.  A function $d:X\times X \rightarrow
[0,\infty]$ is called a \textit{generalized metric} on $X$ if $d$
satisfies the
following:\\
(1) $d(x,y)=0$ if and only if $x=y$;\\
(2) $d(x,y)=d(y,x)$ for all $x,y\in{X}$;\\
(2) $d(x,z)\leq d(x,y)+d(y,z)$ for all $x,y,z\in{X}$.
\begin{thm}\label{t3}  \cite{diaz} Suppose we are given a complete generalized metric space $(X,d)$ and a strictly
contractive mapping $J:X\rightarrow X$, white the Lipshitz
constant $L<1$. If there exists a nonnegative integer $k$ such that  $d(J^{k}x,J^{k+1}x)<\infty$ for some $x\in X$, then the following are true:\\
(1) the sequence ${J^{n}x}$ converges to a fixed point
$x^{\ast}$ of $J$;\\
(2) $x^{\ast}$ is the unique fixed point of $J$ in the
set  $Y=\{y\in{X}:d(J^{k}x,y)<\infty\}$;\\
(3)\phantom{+}$d(y,x^{\ast})\leq \frac{1}{1-L} d(y,Jy)$ for all
$y\in{Y}$.
\end{thm}
 Throughout this paper, we fix a real number
$\beta$ with $0<\beta\leq 1$ and let $\mathbb{K}$ denote either
$\mathbb{R}$ or $\mathbb{C}$. Suppose $E$ is a vector space over
$\mathbb{K}$. A function $\|.\|_{\beta}$: $E\longrightarrow
[0,\infty)$ is called a $\beta$-norm if and only if it satisfies\\
(1) $\|x\|_{\beta}=0$, if and only if $x = 0$;\\
(2) $\|\lambda x\|_{\beta}=|\lambda|^{\beta}\|x\|_{\beta}$ for all $\lambda\in \mathbb{K}$ and all $x \in E$; \\
(3) $\| x+y\|_{\beta}\leq \|x\|_{\beta}+\|y\|_{\beta}$ for all $x,
y$ $\in E$.\section{main results} In the following theorem, by using
an idea of C\v{a}dariu and Radu \cite{1a, Aa}, we  prove the
Hyers-Ulam-Rassias stability of the generalized Pexider functional
equation (\ref{eq11}).
\begin{thm}\label{t21}Let $E$ be a vector space over $\mathbb{K}$ and let
$F$ be a complete $\beta$-normed space over $\mathbb{K}$. Let $K$ be
a finite abelian   subgroup of the  automorphism group of $(E,+)$.
Let $f$: $E\longrightarrow F$ be a mapping for which there exists a
function $\varphi:E\times F\rightarrow [0,\infty)$ and  a constant
$L<1$, such that
\begin{equation}\label{eq21}
\|\frac{1}{|K|}\sum_{k\in{K}} f(x+k\cdot
y)-g(x)-h(y)\|_{\beta}\leq\varphi(x,y)
\end{equation}and
\begin{equation}\label{eq22}
\sum_{k\in{K}} \varphi(x+k\cdot x,y+k\cdot y)\leq(|2
K|)^{\beta}L\varphi(x,y)
\end{equation}
for all $x,y\in E$.  Then, there exists a unique solution $q$:
$E\longrightarrow F$ of the generalierd quadratic functional
equation (\ref{eq12}) and a unique solution $j$: $E\longrightarrow
F$ of the generalized Jensen functional equation (\ref{eq13}) such
that

\begin{equation}\label{eq001}
    \frac{1}{|K|}\sum_{k\in K}j(k\cdot x)=0,
\end{equation}

\begin{equation}\label{eq23}
    \| f(x)-q(x)-j(x)-g(0)-h(0)\|_{\beta}\leq\frac{2}{2^{\beta}}\frac{1}{1-L}\chi(x,x)+\frac{1}{2^{\beta}}\frac{1}{1-L}\psi(x,x),
\end{equation}
\begin{equation}\label{eq24}
    \|
    g(x)-q(x)-j(x)-g(0)\|_{\beta}\leq\varphi(x,0)+\frac{2}{2^{\beta}}\frac{1}{1-L}\chi(x,x)+\frac{1}{2^{\beta}}\frac{1}{1-L}\psi(x,x)
\end{equation}
and \begin{equation}\label{eq25}
    \| h(x)-q(x)-h(0)\|_{\beta}\leq\frac{1}{2^{\beta}}\frac{1}{1-L}\psi(x,x)+\varphi(0,x)
\end{equation}
for all $x\in{E}$, where
$$\chi(x,y)=\frac{|K|}{|K|^{\beta}}\varphi(0,y)+\varphi(x,y)+\varphi(x,0)+\varphi(0,y)$$
$$+\frac{1}{|K|^{\beta}}\sum_{k\in K}[\varphi(k\cdot x,y)+\varphi(k\cdot x,0)]$$
and
$$\psi(x,y)=\frac{|K|}{|K|^{\beta}}\varphi(0,y)+\frac{1}{|K|^{\beta}}\sum_{k\in
K}[\varphi(k\cdot x,y)+\varphi(k\cdot x,0)].$$
\end{thm}
\begin{proof} Letting $y=0$ in (\ref{eq21}), to obtain
\begin{equation}\label{eq26}
    \|f(x)-g(x)-h(0)\|_{\beta}\leq \varphi(x,0)
\end{equation}
for all $x\in E$. By using (\ref{eq26}), (\ref{eq21}) and the
triangle inequality, we get
\begin{equation}\label{eq27}
\|\frac{1}{|K|}\sum_{k\in{K}} f(x+k\cdot
y)-f(x)-(h(y)-h(0))\|_{\beta}\leq \|\frac{1}{|K|}\sum_{k\in{K}}
f(x+k\cdot
y)-g(x)-h(y)\|_{\beta}\end{equation}$$+\|g(x)-f(x)+h(0)\|_{\beta}\leq\varphi(x,y)+\varphi(x,0)$$
for all $x,y\in E.$ Replacing $x$ by $0$ in (\ref{eq21}), we get
\begin{equation}\label{eq28}
    \|\frac{1}{|K|}\sum_{k\in{K}} f(k\cdot y)-g(0)-h(y)\|_{\beta}\leq\varphi(0,y)
\end{equation}
for all $y\in E$. So inequalities (\ref{eq27}), (\ref{eq28}) and the
triangle inequality implies that
\begin{equation}\label{eq28'}
    \|\frac{1}{|K|}\sum_{k\in{K}} f(x+k\cdot y)-f(x)-\frac{1}{|K|}\sum_{k\in{K}} f(k\cdot y)+g(0)+h(0)\|_{\beta}\leq
    \frac{1}{|K|}\sum_{k\in{K}} f(x+k\cdot
y)-f(x)-(h(y)-h(0))\|_{\beta}
\end{equation}
$$+\|\frac{1}{|K|}\sum_{k\in{K}} f(k\cdot y)-h(y)-g(0)\|_{\beta}\leq\varphi(x,y)+\varphi(x,0)+\varphi(0,y)$$
for all $x,y\in E.$ Now, let
\begin{equation}\label{eq0}
\phi(x)=\frac{1}{|K|}\sum_{k\in K}f(k\cdot x)
\end{equation}
 for all $x\in E$.
Then, $\phi$ satisfies
\begin{equation}\label{eq29}
    \frac{1}{|K|}\sum_{k\in K}\phi(k\cdot x)=\phi(x)
\end{equation} for all $x\in E$. Furthermore, in view of  (\ref{eq28'}), (\ref{eq29}) and the triangle inequality, we have
\begin{equation}\label{eq230}
    \|\frac{1}{|K|}\sum_{k'\in{K}} \phi(x+k'\cdot y)-\phi(x)-\phi(y)+g(0)+h(0)\|_{\beta}\end{equation}
    $$=\|\frac{1}{|K|}\sum_{k'\in{K}}\frac{1}{|K|}\sum_{k\in{K}} f(k\cdot x+kk'\cdot y)-\frac{1}{|K|}\sum_{k\in{K}}f(k\cdot x)-\frac{1}{|K|^{2}}\sum_{k,k'\in{K}}f(kk'\cdot y)+g(0)+h(0)\|_{\beta}$$
    $$\leq\frac{1}{|K|^{\beta}}\sum_{k\in{K}}\|\frac{1}{|K|}\sum_{k'\in{K}}f(k\cdot x+k'\cdot y)-f(k\cdot x)-\frac{1}{|K|}\sum_{k'\in{K}}f(k'\cdot y)+g(0)+h(0)\|_{\beta}$$
    $$\leq\frac{1}{|K|^{\beta}}\sum_{k\in K}[\varphi(k\cdot x,y)+\varphi(k\cdot x,0)]+\frac{|K|}{|K|^{\beta}}\varphi(0,y)=\psi(x,y).$$
    Since $K$ is an abelian subgroup, so by using (\ref{eq22}), we get
    \begin{equation}\label{eq231}
\sum_{k\in{K}} \psi(x+k\cdot x,y+k\cdot y)\leq(2|
K|)^{\beta}L\psi(x,y)
\end{equation} for all $x,y\in E.$ Let us consider the set  $X:=\{g:\;E\longrightarrow F\}$  and introduce the \textit{generalized
metric} on $X$ as follows:
\begin{equation}\label{eq232}d(g,h)=\inf\{C\in[0,\infty]:\; \| g(x)-h(x)\|_{\beta}\leq
C\psi(x,x),\;\forall x\in E\}.\end{equation}
Let ${f_n}$ be a Cauchy sequence in $(X,d)$. According to the definition of the Cauchy sequence, for
any given $\varepsilon > 0$, there exists a positive integer $N$ such that
\begin{equation}\label{eq233}
d(f_n,f_m)\leq \varepsilon
\end{equation} for all integer  $m,n$ such that $m\geq N$ and $n\geq N$. That is, by considering the definition of the \textit{generalized metric} $d$
\begin{equation}\label{eq234}
    \|f_m(x)-f_n(x)\|_\beta\leq \varepsilon \psi(x,x)
\end{equation} for all integer  $m,n$ such that $m\geq N$ and $n\geq N$, which implies that ${f_n(x)}$ is a Cauchy sequence in $F$, for any fixed $x\in E.$ Since
$F$ is complete, ${f_n(x)}$ converges in $F$ for each $x$ in $E$. Hence, we can define a function
$f$ : $E\longrightarrow F$ by
\begin{equation}\label{eq235}
    f(x)=\lim_{n\longrightarrow\infty}f_n(x).
\end{equation}
As a similar proof to \cite{man1}, we consider the linear operator  $J:X\rightarrow X$
 such that
\begin{equation}\label{eq236}
(Jh)(x)= \frac{1}{2|K|}\sum_{k\in{K}} h(x+k\cdot x)
\end{equation}
for all  $x\in E$. By induction, we can easily show that
\begin{equation}\label{eq237}
(J^{n} h)(x)=\frac{1}{(2|
K|)^{n}}\sum_{k_{1},...,k_{n}\in{K}}h\left(x+\sum_{i_{j}<i_{j+1},k_{ij}\in{\{k_{1},...,k_{n}\}}}(k_{i_{1}}...
k_{i_{p}})\cdot x\right)
\end{equation}
for all integer $n$. \\First, we assert that $J$ is strictly contractive on $X$. Given $g$ and $h$ in $X$, let $C\in [0,\infty)$  be an arbitrary constant with $d(g,h)\leq
C$, that is,
\begin{equation}\label{eq238}
\| g(x)-h(x)\|_{\beta}\leq C \psi(x,x)
\end{equation}
for all  $x\in E$.   So, it follows from (\ref{eq236}), (\ref{eq231}) and
( \ref{eq238}) we get
\begin{align*}
   \|(J g)(x)-(J h)(x)\|_{\beta}&=\|\frac{1}{2| K|}\sum_{k\in{K}}
g(x+k\cdot x)-\frac{1}{2| K|}\sum_{k\in{K}} h(x+k\cdot
x)\|_{\beta}\\&=\frac{1}{(2|K|)^{\beta}}\|\sum_{k\in{K}}
g(x+k\cdot x)-h(x+k\cdot
x)\|_{\beta}\\&\leq\frac{1}{(2|K|)^{\beta}}\sum_{k\in{K}} \|
g(x+k\cdot x)-h(x+k\cdot
x))\|_{\beta}\\&\leq\frac{1}{(2|K|)^{\beta}}C\sum_{k\in{K}}
\psi(x+k\cdot x,x+k\cdot x)\\&\leq CL \psi(x,x)
\end{align*}
for all  $x\in E$, that is, $d(J g,J h)\leq LC$. Hence, we conclude that  $$d(J g,J h)\leq L d(g,h)$$ for any $g,h\in{X}$. Now, we claim that
\begin{equation}\label{eq239}
d(J(\phi-g(0)-h(0), \phi-g(0)-h(0))< \infty.
\end{equation}
  By
letting
 $y=x$ in (\ref{eq230}),  we obtain
\begin{equation}\label{eq240}
\|(J
(\phi-g(0)-h(0)))(x)-(\phi-g(0)-h(0))(x)\|_{\beta}=\frac{1}{2^{\beta}}
\|\frac{1}{|K|}\sum_{k\in{K}} \phi(x+k\cdot
x)-2\phi(x)+g(0)+h(0)\|_{\beta}\leq \frac{1}{2^{\beta}}\psi(x,x)
\end{equation}
for all $x\in E$,  that is
\begin{equation}\label{eq241}
    d(J (\phi-g(0)-h(0)),\phi-g(0)-h(0))\leq \frac{1}{2^{\beta}}<\infty
\end{equation}
From Theorem 1.1, there exists a fixed point of $J$ which is  a
function $q:E\rightarrow F$ such that $\lim_{n\longrightarrow
\infty}d(J^{n} (\phi-g(0)-h(0)),q)= 0$. Since $d(J^{n}
(\phi-g(0)-h(0)),q)\rightarrow 0$ as $n\rightarrow\infty$, there
exists a sequence $\{C_{n}\}$ such that
$\lim_{n\longrightarrow\infty}C_{n} = 0$
 and $d(J^{n} \phi-g(0)-h(0),q)\leq C_{n}$ for every
$n\in{\mathbb{N}}$. Hence,  from the definition of $d$, we get
\begin{equation}\label{eq242}
\|(J^{n} (\phi-g(0)-h(0)(x)-q(x)\|_{\beta}\leq C_{n}\psi(x,x)
\end{equation}\\
for all $x\in E$. Therefore,
\begin{equation}\label{eq243}
    \lim_{n\rightarrow\infty}\|(J^{n}(\phi-g(0)-h(0))(x)-q(x)\|_{\beta}=0,
\end{equation} for all $x\in E$.\\Now,  if we put $\kappa(x)=\phi(x)-g(0)-h(0)$, by using induction on $n$ we prove the validity of
following inequality
\begin{equation}\label{eq244}
\|\frac{1}{| K|}\sum_{k\in{K}} J^{n}\kappa(x+k\cdot
y)-J^{n}\kappa(x)-J^{n}\kappa(y)\|_{\beta}\leq L^{n}\psi(x,y).
\end{equation}
In view of  the commutativity of $K$  the  inequalities (\ref{eq230}) and (\ref{eq231}) we have
$$ \|\frac{1}{| K|}\sum_{k\in{K}} J
f(x+k\cdot y)-J \kappa(x)-J \kappa(y)\|_{\beta}$$
$$=\|\frac{1}{|
K|}\sum_{k\in{K}}\frac{1}{2| K|}\sum_{k_{1}\in{K}}\kappa(x+k\cdot
y+k_{1}\cdot x+k_{1}k\cdot y)-\frac{1}{2|
K|}\sum_{k_{1}\in{K}}\kappa(x+k_{1}\cdot x) -\frac{1}{2|
K|}\sum_{k_{1}\in{K}}\kappa(y+k_{1}\cdot y)\|_{\beta}$$$$\leq
\frac{1}{(2| K|^{\beta})}\sum_{k_{1}\in{K}}\|\frac{1}{|
K|}\sum_{k\in{K}}\kappa(x+k_{1}\cdot x+k\cdot(y+k_{1}\cdot
y))-\kappa(x+k_{1}\cdot x)-\kappa(y+k_{1}\cdot
y)\|_{\beta}$$$$\leq\frac{1}{(2|
K|^{\beta})}\sum_{k_{1}\in{K}}\psi(x+k_{1}\cdot x,y+k_{1}\cdot
y)\leq \frac{1}{(2| K|)^{\beta}}(2|
K|)^{\beta}L\psi(x,y)=L\psi(x,y).$$ This proves (\ref{eq244}) for
$n=1$.  Now, we assume that (\ref{eq244}) is true for   $n$. By
using the commutativity of $K$, the inequalities (\ref{eq230}),
(\ref{eq231}),  we get $$\displaystyle\|\frac{1}{| K|}\sum_{k\in{K}}
J^{n+1}\kappa(x+k\cdot
y)-J^{n+1}\kappa(x)-J^{n+1}\kappa(y)+g(0)+h(0)\|_{\beta}$$
$$=\|\frac{1}{| K|}\sum_{k\in{K}}\frac{1}{2|
K|}\sum_{k^{'}\in{K}}J^{n}\kappa(x+k\cdot y+k^{'}\cdot x+k^{'}k\cdot
y)$$$$-\frac{1}{2| K|}\sum_{k^{'}\in{K}}J^{n}\kappa(x+k^{'}\cdot
x)-\frac{1}{2| K|}\sum_{k^{'}\in{K}}J^{n}\kappa(y+k^{'}\cdot
y)\|_{\beta}$$
 $$\leq \frac{1}{(2| K|)^{\beta}}\sum_{k^{'}\in{K}}\|\frac{1}{|
K|} \sum_{k\in{K}}J^{n}\kappa(x+k^{'}\cdot x+k\cdot(y+k^{'}\cdot
y)-J^{n}\kappa(x+k'\cdot x)-J^{n}\kappa(y+k'\cdot y)\|_{\beta}$$
$$\leq\frac{1}{(2|
K|)^{\beta}}\sum_{k^{'}\in{K}}L^{n}\psi(x+k^{'}\cdot x,y+k^{'}\cdot
y)\leq L^{n+1}\psi(x,y),$$ which proves (\ref{eq244}) for $n+1$.
Now, by letting $n\rightarrow\infty$, in (\ref{eq244}), we obtain
that $q$ is a solution of equation (\ref{eq12}).  According to the
fixed point theorem (Theorem 1.1, (3)) and  inequality
(\ref{eq241}), we get
\begin{equation}\label{eq245}
d(\phi-g(0)-h(0),q)\leq \frac{1}{1-L}d(J
(\phi-g(0)-h(0)),\phi-g(0)-h(0))\leq\frac{1}{2^{\beta}}\frac{1}{(1-L)}
\end{equation}and so we have
\begin{equation}\label{eq246}
\|\phi(x)-q(x)-g(0)-h(0))\|\leq\frac{1}{2^{\beta}}\frac{1}{(1-L)}\psi(x,x)
\end{equation}for all $x\in E.$ On the other hand if we put
\begin{equation}\label{eq247}
\omega(x)=f(x)-\phi(x)=f(x)-\frac{1}{|K|}\sum_{k\in K}f(k\cdot x)
\end{equation}for all $x\in E,$ it follows from inequalities (\ref{eq28'}), (\ref{eq230}) and the triangle inequality that
\begin{equation}\label{eq248}
    \|\frac{1}{|K|}\sum_{k'\in K}\omega(x+k'\cdot y)-\omega(x)\|_{\beta}
\end{equation}
$$= \|\frac{1}{|K|}\sum_{k'\in K}f(x+k'\cdot y)-\frac{1}{|K|}\sum_{k\in K}\phi( x+k\cdot y)-f(x)+\phi(x)\|_{\beta}$$
$$\leq \|-\frac{1}{|K|}\sum_{k\in K}\phi( x+k\cdot y)+\phi(x)+\phi(y)-g(0)-h(0)\|_{\beta} $$
$$+\|\frac{1}{|K|}\sum_{k'\in K}f(x+k'\cdot y)-f(x)-\frac{1}{|K|}\sum_{k'\in K}f(k'\cdot y)+g(0)+h(0)\|_{\beta}$$
$$\leq\frac{1}{|K|^{\beta}}\sum_{k\in K}[\varphi(k\cdot x,y)+\varphi(k\cdot x,0)]+\frac{|K|}{|K|^{\beta}}\varphi(0,y)+
\varphi(x,y)+\varphi(x,0)+\varphi(0,y)=\chi(x,y)$$ for all $x,y\in
E.$ By using the same definition for $X$ as in the above proof, the
\textit{generalized metric} on $X$
\begin{equation}\label{eq249}d(g,h)=\inf\{C\in[0,\infty]:\; \| g(x)-h(x)\|_{\beta}\leq
C\chi(x,x),\;\forall x\in E\}.\end{equation} and some ideas of
\cite{man1}, we will prove that there exists a unique solution $j$
of equation (\ref{eq13}) such that
\begin{equation}\label{eq250}
    \|\omega(x)-j(x)\|_{\beta}\leq\frac{1}{1-L}\chi(x,x)
\end{equation} for all $x\in E.$\\ First, from (\ref{eq22}) we can easily verify that $\chi(x,y)$ satisfies
\begin{equation}\label{eq251}
\sum_{k\in{K}} \chi(x+k\cdot x,y+k\cdot y)\leq(2|
K|)^{\beta}L\chi(x,y)
\end{equation}
 Let us consider the function $T$ $:X\rightarrow X$ defined by
\begin{equation}\label{eq252}
    (T h )(x)=\frac{1}{|2 K|}\sum_{k\in{K}}h(x+k\cdot x)
\end{equation}
for all $x\in E$. Given $g,h\in X$ and $C\in [0,\infty]$ such that
$d(g,h)\leq C$, so we get $$\|(T g )(x)-(T h
)(x)\|_{\beta}=\|\frac{1}{|2 K|}\sum_{k\in{K}}g(x+k\cdot
x)-\frac{1}{|2 K|}\sum_{k\in{K}}h(x+k\cdot x)\|_{\beta}$$
$$=\frac{1}{|2 K|^{\beta}}\|\sum_{k\in{K}}[g(x+k\cdot
x)-h(x+k\cdot x)]\|_{\beta}$$
$$\leq\frac{1}{|2
K|^{\beta}}\sum_{k\in{K}}\| g(x+k\cdot x)-h(x+k\cdot x)\|_{\beta}
\leq CL\chi(x,x)$$ for all $x\in E$. Hence, we see that $d(T g ,T
h)\leq L d(g,h)$ for all $g,h\in X$.  So $T$ is a strictly
contractive operator.\\ Putting $y=x$ in (\ref{eq248}), we have
\begin{equation}\label{eq253}
\|\frac{1}{|2K|}\sum_{k\in K}\omega(x+k\cdot
x)-\frac{1}{2}\omega(x)\|_{\beta}\leq\frac{1}{2^{\beta}} \chi(x,x)
\end{equation}for all $x\in E.$ So by  triangle inequality, we get $$\|\frac{1}{|2K|}\sum_{k\in K}\omega(x+k\cdot
x)-\omega(x)\|_{\beta}\leq\frac{2}{2^{\beta}} \chi(x,x)$$ for all $x\in E,$ that is,
\begin{equation}\label{eq254}
d(T \omega,\omega)\leq \frac{2}{2^{\beta}}.
\end{equation}
From the fixed point theorem (Theorem 1.1), it follows that there
exits a fixed point $j$ of $T$ in $X$ such that
\begin{equation}\label{eq255}
j(x)=\lim_{n\rightarrow\infty}\frac{1}{|2
K|^{n}}\sum_{k_{1},...,k_{n}\in{K}}\omega\left(x+\sum_{i_{j}<i_{j+1},k_{ij}\in{\{k_{1},...,k_{n}\}}}[(k_{i_{1}})\cdot\cdot\cdot
(k_{i_{p}})]\cdot x\right)
\end{equation} for all $x\in E$ and
\begin{equation}\label{eq256}
d(\omega,j)\leq\frac{1}{1-L}d(T\omega,\omega).
\end{equation}So, it follows from (\ref{eq254}) and (\ref{eq256})  that
\begin{equation}\label{eq257}
\|\omega(x)-j(x)\|_{\beta}\leq\frac{2}{2^{\beta}}\frac{1}{1-L}\chi(x,x)
\end{equation}for all $x\in E$.\\
By the same reasoning as in the above proof, one can show by
induction that
\begin{equation}\label{eq258}
    \|\frac{1}{| K|}\sum_{k\in{K}}
T^{n}\omega(x+k\cdot y)-T^{n}\omega(x)\|_{\beta}\leq L^{n}\chi(x,y)
\end{equation}
for all $x,y\in E$ and for all $n\in {\mathbb{N}}.$ Letting
$n\rightarrow\infty$ in  (\ref{eq258}), we get  that $j$ is a
solution of the generalized Jensen functional equation (\ref{eq13}).
\\From  (\ref{eq0}), (\ref{eq246}) (\ref{eq247}), (\ref{eq257})  and the triangle inequality, we obtain
\begin{equation}\label{eq259}
    \| f(x)-q(x)-j(x)-g(0)-h(0)\|_{\beta}\leq\frac{2}{2^{\beta}}\frac{1}{1-L}\chi(x,x)+\frac{1}{2^{\beta}}\frac{1}{1-L}\psi(x,x),
\end{equation}
\begin{equation}\label{eq260}
    \|
    g(x)-q(x)-j(x)-g(0)\|_{\beta}\leq\varphi(x,0)+\frac{2}{2^{\beta}}\frac{1}{1-L}\chi(x,x)+\frac{1}{2^{\beta}}\frac{1}{1-L}\psi(x,x)
\end{equation}
and \begin{equation}\label{eq261}
    \| h(x)-q(x)-h(0)\|_{\beta}\leq\frac{1}{2^{\beta}}\frac{1}{1-L}\psi(x,x)+\varphi(0,x)
\end{equation}
for all $x\in{E}$.\\
Finally, in the following we will verify that the solution $j$
satisfies the condition
\begin{equation}\label{eq262}
\frac{1}{| K|}\sum_{k\in{K}} j(k\cdot x)=0
\end{equation}for all $x\in E$
 and we will prove the uniqueness of the
solutions $q$ and $j$ which satisfy the inequalities (\ref{eq259})
(\ref{eq260}) and (\ref{eq261}).\\Due to definition of $\omega$, we
get $\frac{1}{| K|}\sum_{k\in{K}} \omega(k\cdot x)=0$ for all $x\in
E$, so we get $\frac{1}{| K|}\sum_{k\in{K}} T\omega(k\cdot x)=0$,
$\frac{1}{| K|}\sum_{k\in{K}} T^{2}\omega(k\cdot x)=0,...,$
$\frac{1}{| K|}\sum_{k\in{K}} T^{n}\omega(k\cdot x)=0$. So, by
letting $n\longrightarrow \infty$, we obtain the ralation
(\ref{eq262}).\\Now, according  to (\ref{eq261}) and (\ref{eq22}) we
get by induction that
\begin{equation}\label{eq263}
    \| J^{n}(h-h(0))(x)-q(x)\|_{\beta}\leq
    L^{n}[\frac{1}{2^{\beta}}\frac{1}{1-L}\psi(x,x)+\varphi(0,x)]
\end{equation}
for all $x\in{E}$ and for all $n\in \mathbb{N}.$ So, by letting
$n\longrightarrow \infty,$ we get \begin{equation}
 \lim_{n\longrightarrow\infty}J^{n}(h-h(0))(x)=q(x)
\end{equation} for all $x\in{E}$, which proves the uniqueness of
$q$.\\In a similar way, by induction we obtain
\begin{equation}\label{eq264}
    \| \Lambda^{n}(f-q-h(0)-g(0))(x)-j(x)\|_{\beta}\leq
    L^{n}[\frac{1}{1-L}\chi(x,x)+\frac{1}{2^{\beta}}\frac{1}{1-L}\psi(x,x)]
\end{equation}
for all $x\in{E}$ and for all $n\in \mathbb{N},$ where $$\Lambda l(x)=\frac{1}{|K|}\sum_{k\in K}l(x+k\cdot x).$$ Consequently, we
have
\begin{equation}\label{eq265}
 \lim_{n\longrightarrow\infty}\Lambda^{n}(f-q-h(0)-g(0))(x)=j(x)
\end{equation} for all $x\in{E}$. This proves the uniqueness of the
function $j$ and this completes the proof of theorem.\\\\In the
following, we will investigate some special cases of Theorem 2.1,
with the new weaker conditions.
\end{proof}
\begin{cor}\label{t21}Let $E$ be a vector
space over $\mathbb{K}$. Let $K$ be a finite abelian subgroup of the
automorphism group of $(E,+)$, Let
$\alpha=\frac{\log(|K|)}{\log(2)}$. Fix a nonnegative real number
$\beta$ such that $\frac{\alpha}{\alpha+1}<\beta<1$ and choose a
number $p$ with $0<p<\beta+(\beta-1)\alpha$ and let $F$ be a
complete $\beta$-normed space over $\mathbb{K}$.  If a function $f$:
$E\longrightarrow F$ satisfies
\begin{equation}\label{eq266}
\|\frac{1}{|K|}\sum_{k\in{K}} f(x+k\cdot
y)-g(x)-h(y)\|_{\beta}\leq\theta(\|x\|^{p}+\|y\|^{p})
\end{equation}and $\|x+k\cdot x\|\leq 2\|x\|$, for all $k\in K,$
for all $x,y\in E$ and for some $\theta> 0$,  then there exists a
unique solution $q$: $E\longrightarrow F$ of the generalierd
quadratic functional equation (\ref{eq12}) and a unique solution
$j$: $E\longrightarrow F$ of the generalized Jensen functional
equation (\ref{eq13}) such that

\begin{equation}\label{eq268}
    \frac{1}{|K|}\sum_{k\in K}j(k\cdot x)=0,
\end{equation}

\begin{equation}\label{eq269}
    \| f(x)-q(x)-j(x)-g(0)-h(0)\|_{\beta}\leq\frac{\theta}{2^{\beta}}\frac{(2|K|)^{\beta}}{(2|K|)^{\beta}-2^{p}|K|}
    [\frac{|K|}{|K|^{\beta}}(6+6. 3^{p})+8]\|x\|^{p}
\end{equation}
\begin{equation}\label{eq270}
    \|
    g(x)-q(x)-j(x)-g(0)\|_{\beta}\leq\frac{\theta}{2^{\beta}}\frac{(2|K|)^{\beta}}{(2|K|)^{\beta}-2^{p}|K|}
    [\frac{|K|}{|K|^{\beta}}(6+6. 3^{p})+8]\|x\|^{p}+\theta\|x\|^{p}
\end{equation}
and \begin{equation}\label{eq271}
    \|
    h(x)-q(x)-h(0)\|_{\beta}\leq \frac{\theta}{2^{\beta}}\frac{(2|K|)^{\beta}}{(2|K|)^{\beta}-2^{p}|K|}
    [\frac{|K|}{|K|^{\beta}}(2+2. 3^{p})]\|x\|^{p}+\theta\|x\|^{p}
\end{equation}
for all $x\in{E}$.
\end{cor}\begin{proof}The proof follows from Theorem 2.1 by taking $$\varphi(x,y)=\theta(\|x\|^{p}+\|y\|^{p})$$ for all $x,y\in E.$ Then we can choose $L=\frac{2^{p}}{2^{\beta}}\frac{|K|^{\beta}}{|K|}$ and we get the desired result.\end{proof}
\begin{cor} ($K=\{I\}$) Let $E$ be a vector
space over $\mathbb{K}$.  Fix a nonnegative real number $\beta$ less
than $1$ and choose a number $p$ with $0<p<1$ and let $F$ be a
complete $\beta$-normed space over $\mathbb{K}$.  If a function
$(f,g,h)$: $E\longrightarrow F$ satisfies
\begin{equation}\label{eq272}
\| f(x+y)-g(x)-h(y)\|_{\beta}\leq\theta(\|x\|^{p}+\|y\|^{p})
\end{equation}
for all $x,y\in E$ and for some $\theta> 0$,  then there exists an
unique additive function $a$: $E\longrightarrow F$ such that
\begin{equation}\label{eq273}
    \| f(x)-a(x)-g(0)-h(0)\|_{\beta}\leq\frac{\theta}{2^{\beta}}\frac{2^{\beta}}{2^{\beta}-2^{p}}
    [14+6. 3^{p}]\|x\|^{p},
\end{equation}
\begin{equation}\label{eq274}
    \|
    g(x)-a(x)-g(0)\|_{\beta}\leq\frac{\theta}{2^{\beta}}\frac{2^{\beta}}{2^{\beta}-2^{p}}
    [14+6. 3^{p}]\|x\|^{p}+\theta\|x\|^{p}
\end{equation}
and \begin{equation}\label{eq275}
    \| h(x)-a(x)-h(0)\|_{\beta}\leq\frac{\theta}{2^{\beta}}\frac{2^{\beta}}{2^{\beta}-2^{p}}
    [2+2. 3^{p}]\|x\|^{p}+\theta\|x\|^{p}
\end{equation}
for all $x\in{E}$.
\end{cor}
\begin{cor}  ($K=\{I,\sigma\}$) Let $E$ be a vector
space over $\mathbb{K}$. Let $K=\{I,\sigma\}$ where $\sigma$ is an
volution of $E$ ($\sigma(x+y)=\sigma(x)+\sigma(y)$ and
$\sigma(\sigma(x))=x$ for all $x,y\in E$). Fix a nonnegative real
number $\beta$ such that $\frac{1}{2}<\beta<1$ and choose a number
$p$ with $0<p<2\beta-1$ and let $F$ be a complete $\beta$-normed
space over $\mathbb{K}$.  If a function $(f,g,h)$: $E\longrightarrow
F$ satisfies
\begin{equation}\label{eq276}
\| f(x+
y)+f(x+\sigma(y))-g(x)-h(y)\|_{\beta}\leq\theta(\|x\|^{p}+\|y\|^{p})
\end{equation}and $\|x+\sigma( x)\|\leq 2\|x\|$,
for all $x,y\in E$ and for some $\theta> 0$,  then there exists a
unique solution $q$: $E\longrightarrow F$ of the generalierd
quadratic functional equation
\begin{equation}\label{eq277}
f(x+y)+f(x+\sigma(y))=2f(x)+2f(y),\;\;x,y\in E
\end{equation}
 and a unique solution $j$: $E\longrightarrow F$ of the
generalized Jensen functional equation \begin{equation}\label{eq278}
f(x+y)+f(x+\sigma(y))=2f(x),\;\;x,y\in E
\end{equation} such that
\begin{equation}\label{eq279}
    j(\sigma(x))=-j(x),
\end{equation}

\begin{equation}\label{eq280}
    \| f(x)-q(x)-j(x)-g(0)-h(0)\|_{\beta}\leq\frac{\theta}{2^{\beta}}\frac{4^{\beta}}{4^{\beta}-2^{p+1}}
    [\frac{2}{2^{\beta}}(6+6. 3^{p})+8]\|x\|^{p}
\end{equation}
\begin{equation}\label{eq281}
    \|
    g(x)-q(x)-j(x)-g(0)\|_{\beta}\leq\frac{\theta}{2^{\beta}}\frac{4^{\beta}}{4^{\beta}-2^{p+1}}
    [\frac{2}{2^{\beta}}(6+6. 3^{p})+8]\|x\|^{p}+\theta\|x\|^{p}
\end{equation}
and \begin{equation}\label{eq282}
    \|
    h(x)-q(x)-h(0)\|_{\beta}\leq \frac{\theta}{2^{\beta}}\frac{4^{\beta}}{4^{\beta}-2^{p+1}}
    [\frac{2}{2^{\beta}}(2+2. 3^{p})]\|x\|^{p}+\theta\|x\|^{p}
\end{equation}
for all $x\in{E}$.
\end{cor}
\begin{cor}\label{t21}Let $E$ be a vector space over $\mathbb{K}$ and let
$F$ be a complete $\beta$-normed space over $\mathbb{K}$. Let $f$:
$E\longrightarrow F$ be a mapping for which there exists a function
$\varphi:E\times F\rightarrow [0,\infty)$ and  a constant $L<1$,
such that
\begin{equation}\label{eq283}
\| f(x+ y)+f(x+\sigma(y))-g(x)-h(y)\|_{\beta}\leq\varphi(x,y)
\end{equation}and
\begin{equation}\label{eq284}
 \varphi(2x,2 y)+\varphi(x+\sigma(x), y+\sigma(y))\leq 4^{\beta}L\varphi(x,y)
\end{equation}
for all $x,y\in E$.  Then, there exists a unique solution $q$:
$E\longrightarrow F$ of the generalierd quadratic functional
equation (\ref{eq277}) and a unique solution $j$: $E\longrightarrow
F$ of the generalized Jensen functional equation (\ref{eq278}) such
that

\begin{equation}\label{eq285}
    j(\sigma( x))=-j(x),
\end{equation}

\begin{equation}\label{eq286}
    \| f(x)-q(x)-j(x)-g(0)-h(0)\|_{\beta}\leq\frac{2}{2^{\beta}}\frac{1}{1-L}\chi(x,x)+\frac{1}{2^{\beta}}\frac{1}{1-L}\psi(x,x),
\end{equation}
\begin{equation}\label{eq287}
    \|
    g(x)-q(x)-j(x)-g(0)\|_{\beta}\leq\varphi(x,0)+\frac{2}{2^{\beta}}\frac{1}{1-L}\chi(x,x)+\frac{1}{2^{\beta}}\frac{1}{1-L}\psi(x,x)
\end{equation}
and \begin{equation}\label{eq288}
    \| h(x)-q(x)-h(0)\|_{\beta}\leq\frac{1}{2^{\beta}}\frac{1}{1-L}\psi(x,x)+\varphi(0,x)
\end{equation}
for all $x\in{E}$, where
$$\chi(x,y)=\frac{2}{2^{\beta}}\varphi(0,y)+\varphi(x,y)+\varphi(x,0)+\varphi(0,y)$$
$$+\frac{1}{2^{\beta}}[\varphi( x,y)+\varphi( \sigma(x),y)+\varphi( x,0)+\varphi( \sigma(x),0)]$$
and
$$\psi(x,y)=\frac{2}{2^{\beta}}\varphi(0,y)+\frac{1}{2^{\beta}}[\varphi( x,y)+\varphi( \sigma(x),y)+\varphi( x,0)+\varphi( \sigma(x),0)].$$
\end{cor}
\begin{cor}\label{t21}Let $E$ be a vector space over $\mathbb{K}$ and let
$F$ be a complete $\beta$-normed space over $\mathbb{K}$. Let $f$:
$E\longrightarrow F$ be a mapping for which there exists a function
$\varphi:E\times F\rightarrow [0,\infty)$ and  a constant $L<1$,
such that
\begin{equation}\label{eq289}
\| f(x+ y)-g(x)-h(y)\|_{\beta}\leq\varphi(x,y)
\end{equation}and
\begin{equation}\label{eq290}
 \varphi(2x,2y)\leq 2^{\beta}L\varphi(x,y)
\end{equation}
for all $x,y\in E$.  Then, there exists an unique additive function
$a$: $E\longrightarrow F$  such that

\begin{equation}\label{eq291}
    \| f(x)-a(x)-g(0)-h(0)\|_{\beta}\leq\frac{2}{2^{\beta}}\frac{1}{1-L}\chi(x,x)+\frac{1}{2^{\beta}}\frac{1}{1-L}\psi(x,x),
\end{equation}
\begin{equation}\label{eq292}
    \|
    g(x)-a(x)-g(0)\|_{\beta}\leq\varphi(x,0)+\frac{2}{2^{\beta}}\frac{1}{1-L}\chi(x,x)+\frac{1}{2^{\beta}}\frac{1}{1-L}\psi(x,x)
\end{equation}
and \begin{equation}\label{eq293}
    \| h(x)-a(x)-h(0)\|_{\beta}\leq\frac{1}{2^{\beta}}\frac{1}{1-L}\psi(x,x)+\varphi(0,x)
\end{equation}
for all $x\in{E}$, where
$$\chi(x,y)=\varphi(0,y)+\varphi(x,y)+\varphi(x,0)+\varphi(0,y)+\varphi( x,y)+\varphi( x,0)$$ and
$$\psi(x,y)=\varphi(0,y)+\varphi( x,y)+\varphi( x,0).$$
\end{cor}

 Elhoucien Elqorachi,\\ Department of Mathematics,\\
Faculty of Sciences, Ibn Zohr University, Agadir,
Morocco,\\
E-mail: elqorachi@hotmail.com\\\\\\
John Michael Rassias, \\National and Capodistrian University of
Athens\\ Pedagogical Department E.E. Section of Mathematics and
Informatics 4,\\ Agamemnonos Str., Aghia Paraskevi, Attikis 15342,
Greece " .
\\
E-mail: loannis.Rassias@primedu.uoa.gr; jrassias@primedu.uoa.gr;
jrass@otenet.gr\\\\\\Bouikhalene Belaid\\Department of
Mathematics,\\ University Sultan Moulay Slimane, Faculty of
Technical Sciences, \\B\'eni-Mellal, Morocco
\\E-mail: bbouikhalene@yahoo.fr

\end{document}